\documentclass[10pt,letterpaper]{article}
\usepackage[utf8]{inputenc}
\usepackage[english]{babel}
\usepackage{amsmath}
\usepackage{amsfonts}
\usepackage{amssymb}
\usepackage{graphicx}
\usepackage{mathrsfs}
\usepackage{upref,amsthm,amsxtra,exscale}
\usepackage{cite}
\usepackage[colorlinks=true,urlcolor=blue,
citecolor=red,linkcolor=blue,linktocpage,pdfpagelabels,
bookmarksnumbered,bookmarksopen]{hyperref}

\newtheorem{theorem}{Theorem}[section]
\newtheorem{corollary}[theorem]{Corollary}
\newtheorem{lemma}[theorem]{Lemma}
\newtheorem{proposition}[theorem]{Proposition}
\newtheorem{definition}[theorem]{Definition}
\newtheorem{example}[theorem]{Example}
\newtheorem{remark}[theorem]{Remark}

\numberwithin{equation}{section}

\author{Mónica Clapp\footnote{M. Clapp was partially supported by UNAM-DGAPA-PAPIIT grant IN100718 (Mexico) and CONACYT grant A1-S-10457 (Mexico).} \quad and \quad Filomena Pacella\footnote{F. Pacella was partially supported by PRIN 2015 (Italy) and INDAM-GNAMPA (Italy).}}
\title{Existence of nonradial positive and nodal solutions to a critical Neumann problem in a cone}
\date{\today}

\begin{document}

\maketitle

\begin{abstract}
We study the critical Neumann problem
\begin{equation*}
\begin{cases}
-\Delta u = |u|^{2^*-2}u &\text{in }\Sigma_\omega,\\
\quad\frac{\partial u}{\partial\nu}=0 &\text{on }\partial\Sigma_\omega,
\end{cases}
\end{equation*}
in the unbounded cone $\Sigma_\omega:=\{tx:x\in\omega\text{ and }t>0\}$, where $\omega$ is an open connected subset of the unit sphere $\mathbb{S}^{N-1}$ in $\mathbb{R}^N$ with smooth boundary, $N\geq 3$ and $2^*:=\frac{2N}{N-2}$. We assume that some local convexity condition at the boundary of the cone is satisfied. 

If $\omega$ is symmetric with respect to the north pole of $\mathbb{S}^{N-1}$, we establish the existence of a nonradial sign-changing solution.

On the other hand, if the volume of the unitary bounded cone $\Sigma_\omega\cap B_1(0)$ is large enough (but possibly smaller than half the volume of the unit ball $B_1(0)$ in $\mathbb{R}^N$), we establish the existence of a positive nonradial solution. 
\medskip

\medskip

\end{abstract}

\section{Introduction} \label{sec:introduction}

We consider the Neumann problem
\begin{equation} \label{eq:1}
\begin{cases}
-\Delta u = |u|^{2^*-2}u &\text{in }\Sigma_\omega,\\
\quad\frac{\partial u}{\partial\nu}=0 &\text{on }\partial\Sigma_\omega,
\end{cases}
\end{equation}
in the unbounded cone $\Sigma_\omega:=\{tx:x\in\omega\text{ and }t>0\}$, where $\omega$ is an open connected subset of the unit sphere $\mathbb{S}^{N-1}$ in $\mathbb{R}^N$ with smooth boundary, $N\geq 3$, and $2^*:=\frac{2N}{N-2}$ is the critical Sobolev exponent.

It is well known that, if $\omega=\mathbb{S}^{N-1}$, i.e., if $\Sigma_\omega$ is the whole space $\mathbb{R}^{N}$, then the only positive solutions to the critical problem
\begin{equation} \label{eq:entire}
-\Delta w = |w|^{2^*-2}w,\qquad w\in D^{1,2}(\mathbb{R}^N),
\end{equation} 
are the rescalings and translations of the standard bubble $U$ defined in \eqref{eq:bubble}. Moreover, they are the only nontrivial radial solutions to \eqref{eq:entire}, up to sign. It is immediately deduced that, up to sign, the restriction of the bubbles \eqref{eq:bubble_rescaled} to $\Sigma_\omega$ are the only nontrivial radial solutions of \eqref{eq:1} in any cone; see Proposition \ref{prop:radial_energy}. In addition, if the cone $\Sigma_\omega$ is convex, it was shown in \cite[Theorem 2.4]{lpt} that these are the only positive solutions to \eqref{eq:1}. The convexity property of the cone is crucial in the proof of this result, and it is strongly related to a relative isoperimetric inequality obtained in \cite{lp}.

The aim of this paper is to establish the existence of nonradial solutions to \eqref{eq:1}, both positive and sign-changing. As mentioned above, the positive ones can only exist in nonconvex cones. On the other hand, nodal radial solutions to \eqref{eq:1} do not exist, as this would imply the existence of a nontrivial solution to problem \eqref{eq:2} in the bounded cone $\Lambda_\omega:=\{tx:x\in\omega\text{ and }t\in(0,1)\}$, which is impossible because of the Pohozhaev identity \eqref{eq:pohozhaev} and the unique continuation principle. 

For the problem \eqref{eq:entire} in $\mathbb{R}^N$ various types of sign-changing solutions are known to exist; see \cite{d,dmpp,c,fp}. In particular, a family of entire nodal solutions, which are invariant under certain groups of linear isometries of $\mathbb{R}^N$, were exhibited in \cite{c}. These solutions arise as blow-up profiles of symmetric minimizing sequences for the critical equation in a ball, and are obtained through a fine analysis of the concentration behavior of such sequences.

Here we use some ideas from \cite{c} to produce sign-changing solutions to \eqref{eq:1}, but we exploit a different kind of symmetry. Our main result shows that, if $\omega$ is symmetric with respect to the north pole of $\mathbb{S}^{N-1}$ and if the cone $\Sigma_\omega$ has a point of convexity in the sense of Definition \ref{def:pc}, then the problem \eqref{eq:1} has an axially antisymmetric least energy solution, which is nonradial and changes sign; see Theorem \ref{thm:existence_nodal}. As far as we know, this is the first existence result of a nodal solution to \eqref{eq:1}.

Next, we investigate the existence of positive nonradial solutions. In this case we do not require the cone to have any particular symmetry. We establish the existence of a positive nonradial solution to \eqref{eq:1} under some conditions involving the local convexity of $\Sigma_\omega$ at a boundary point and the measure of the bounded cone $\Lambda_\omega$; see Corollary \ref{cor:nonradial} and Theorem \ref{thm:nonradial}. We refer to Section \ref{sec:positive} for the precise statements and further remarks.

\section{A nonradial sign-changing solution}
\label{sec:nodal}

If $\Omega$ is a domain in $\mathbb{R}^N$ we consider the Sobolev space
$$D^{1,2}(\Omega):=\{u\in L^{2^*}(\Omega):\nabla u\in L^2(\Omega,\mathbb{R}^N)\}$$
with the norm
$$\|u\|_{\Omega}^2:=\int_{\Omega}|\nabla u|^2.$$
We denote by $J_\Omega:D^{1,2}(\Omega)\to\mathbb{R}$ the functional given by
$$J_\Omega(u):=\frac{1}{2}\int_{\Omega}|\nabla u|^2 - \frac{1}{2^*}\int_{\Omega}|u|^{2^*},$$
and its Nehari manifold by
$$\mathcal{N}(\Omega):=\left\{u\in D^{1,2}(\Omega):u\neq 0\;\text{ and }\;\int_{\Omega}|\nabla u|^2 = \int_{\Omega}|u|^{2^*}\right\}.$$
For $u\in D^{1,2}(\Omega)\smallsetminus\{0\}$ let $t_u\in (0,\infty)$ be such that $t_uu\in \mathcal{N}(\Omega)$. Then,
\begin{equation} \label{eq:JQ}
J_\Omega(t_uu)=\frac{1}{N}[Q_\Omega(u)]^\frac{N}{2},\qquad\text{where }\; Q_\Omega(u):=\frac{\int_{\Omega}|\nabla u|^2}{\left(\int_{\Omega}|u|^{2^*}\right)^{2/2^*}}.
\end{equation}
Hence,
\begin{equation} \label{eq:inf}
c_\Omega:=\inf_{u\in\mathcal{N}(\Omega)}J_{\Omega}(u)=\inf_{u\in D^{1,2}(\Omega)\smallsetminus\{0\}}\frac{1}{N}[Q_\Omega(u)]^\frac{N}{2}.
\end{equation}
We set $c_\infty:=c_{\mathbb{R}^N}$. It is well known that this infimum is attained at the function
\begin{equation} \label{eq:bubble}
U(x)=a_N\left(\frac{1}{1+|x|^2}\right)^\frac{N-2}{2},\qquad a_N:=(N(N-2))^\frac{N-2}{4},
\end{equation}
which is called the standard bubble, and at every rescaling and translation of it, and that
$$c_\infty=J_{\mathbb{R}^N}(U)=\frac{1}{N}S^\frac{N}{2},$$
where $S$ is the best constant for the Sobolev embedding $D^{1,2}(\mathbb{R}^N)\hookrightarrow L^{2^*}(\mathbb{R}^N)$.

Let $\mathbb{S}^{N-1}$ be the unit sphere in $\mathbb{R}^N$ and let  $\omega$ be a smooth domain in $\mathbb{S}^{N-1}$ with nonempty boundary, i.e., $\omega$ is connected and open in $\mathbb{S}^{N-1}$ and its boundary $\partial\omega$ is a smooth $(N-2)$-dimensional submanifold of $\mathbb{S}^{N-1}$. The nontrivial solutions to the Neumann problem \eqref{eq:1} in the unbounded cone 
$$\Sigma_\omega:=\{tx:x\in\omega\text{ and }t>0\}$$
are the critical points of $J_{\Sigma_\omega}$ on $\mathcal{N}(\Sigma_\omega)$.

To produce a nonradial sign-changing solution for \eqref{eq:1} we introduce some symmetries. We write a point in $\mathbb{R}^N$ as $x=(x',x_N)\in\mathbb{R}^{N-1}\times\mathbb{R}$, and consider the reflection $\varrho(x',x_N):=(-x',x_N)$. Then, a subset $X$ of $\mathbb{R}^N$ will be called $\varrho$\textbf{-invariant} if $\varrho x\in X$ for every $x\in X$, and a function $u:X\to\mathbb{R}$ will be called $\varrho$\textbf{-equivariant} if
$$u(\varrho x)=-u(x)\quad\forall x\in X.$$
Note that every nontrivial $\varrho$-equivariant function is nonradial and changes sign. 

Throughout this section we will assume that $\omega$ is $\varrho$-invariant. Note that $(0,\pm 1)\not\in\partial\omega$ because $\partial\omega$ is smooth. Hence, $\varrho x\neq x$ for every $x\in\partial\Sigma_\omega\smallsetminus\{0\}$. Our aim is to show that \eqref{eq:1} has a $\varrho$-equivariant solution. We set
$$D^{1,2}_\varrho(\Sigma_\omega):=\{u\in D^{1,2}(\Sigma_\omega):u\text{ is }\varrho\text{-equivariant}\},$$
$$\mathcal{N}^\varrho(\Sigma_\omega):=\{u\in\mathcal{N}(\Sigma_\omega):u\text{ is }\varrho\text{-equivariant}\}$$
and
\begin{equation} \label{eq:inf_rho}
c^\varrho_{\Sigma_\omega}:=\inf_{u\in\mathcal{N}^\varrho(\Sigma_\omega)}J_{\Sigma_\omega}(u)=\inf_{u\in D^{1,2}_\varrho(\Sigma_\omega)\smallsetminus\{0\}}\frac{1}{N}[Q_{\Sigma_\omega}(u)]^\frac{N}{2}.
\end{equation} 

Define 
$$\Lambda_{\omega}:=\{tx:x\in\omega\text{ and }0<t<1\}$$
and set $\Gamma_1:=\partial\Lambda_{\omega}\smallsetminus\overline{\omega}$. In $\Lambda_\omega$ we consider the mixed boundary value problem
\begin{equation} \label{eq:2}
\begin{cases}
-\Delta u = |u|^{2^*-2}u &\text{in }\Lambda_{\omega},\\
\quad u=0 &\text{on } \omega,\\
\quad\frac{\partial u}{\partial\nu}=0 &\text{on }\Gamma_1,
\end{cases}
\end{equation}
We point out that \eqref{eq:2} does not have a nontrivial solution. Indeed, by the well known Pohozhaev identity, a solution to \eqref{eq:2} must satisfy
\begin{equation} \label{eq:pohozhaev}
\frac{1}{2^*}\int_{\Gamma_1}|u|^{2^*}(s\cdot\nu)\mathrm{d}s = - \frac{1}{2}\int_{\omega}\frac{\partial u}{\partial\nu}(s\cdot\nu)\mathrm{d}s.
\end{equation}
As $s\cdot\nu=0$ for every $s\in\Gamma_1$ and $s\cdot\nu>0$ for every $s\in\omega$, we conclude that $\frac{\partial u}{\partial\nu}$ vanishes on $\omega$. Therefore, the trivial extension of $u$ to the infinite cone $\Sigma_\omega$ solves \eqref{eq:1}, contradicting the unique continuation principle.

Let $V(\Lambda_\omega)$ be the space of functions in $D^{1,2}(\Lambda_{\omega})$ whose trace vanishes on $\omega$. Note that $V(\Lambda_\omega)\subset D^{1,2}(\Sigma_{\omega})$ via trivial extension.
Let $J_{\Lambda_\omega}:V(\Lambda_\omega)\to\mathbb{R}$ be the restriction of $J_{\Sigma_\omega}$ to $V(\Lambda_\omega)$ and set
$$\mathcal{N}^\varrho(\Lambda_\omega):=\mathcal{N}^\varrho(\Sigma_\omega)\cap V(\Lambda_{\omega})\qquad\text{and}\qquad c^\varrho_{\Lambda_\omega}:=\inf_{u\in\mathcal{N}^\varrho(\Lambda_\omega)}J_{\Lambda_\omega}(u).$$ 

To produce a sign-changing solution for the problem \eqref{eq:1} we will study the concentration behavior of $\varrho$-equivariant minimizing sequences for \eqref{eq:2}. We start with the following lemmas.

\begin{lemma} \label{lem:infimum}
\;$0<c^\varrho_{\Lambda_\omega} = c^\varrho_{\Sigma_\omega} \leq c_\infty.$
\end{lemma}

\begin{proof}
It is shown in \cite[Theorem 2.1]{lpt} that $c^\varrho_{\Lambda_\omega}>0$.

Since $\mathcal{N}^\varrho(\Lambda_\omega)\subset\mathcal{N}^\varrho(\Sigma_\omega)$, we have that $c^\varrho_{\Lambda_\omega} \geq c^\varrho_{\Sigma_\omega}$. To prove the opposite inequality, let $\varphi_k\in\mathcal{N}^\varrho(\Sigma_\omega)\cap\mathcal{C}^\infty(\overline{\Sigma}_\omega)$ be such that $\varphi_k$ has compact support and $J(\varphi_k)\to c^\varrho_{\Sigma_\omega}$ as $k\to\infty$. Then, we may choose $\varepsilon_k >0$ such that the support of $\widetilde{\varphi}_k(x) := \varepsilon_k^{-(N-2)/2} \varphi_k (\varepsilon_k^{-1}x)$ is contained in $\overline{\Lambda}_\omega\smallsetminus\overline{\omega}$. Thus, $\widetilde{\varphi}_k \in \mathcal{N}^\varrho(\Lambda_\omega)$ and, hence,
\begin{equation*}
c^\varrho_{\Lambda_\omega}\leq J(\widetilde{\varphi}_k) = J(\varphi_k) \quad \text{for all } k.
\end{equation*}
Letting $k \to \infty$ we conclude that $c^\varrho_{\Lambda_\omega} \leq c^\varrho_{\Sigma_\omega}$.

To prove that $c^\varrho_{\Sigma_\omega} \leq c_\infty$ we fix a point $\xi\in\partial\Sigma_\omega\smallsetminus \{0\}$ and a sequence of positive numbers $\varepsilon_k\to 0$, and we set $\Sigma_k:=\varepsilon_k^{-1}(\Sigma_\omega -\xi)$. Since $\partial\Sigma_\omega\smallsetminus \{0\}$ is smooth, the limit of the sequence of sets $(\Sigma_k)$ is the half-space
\begin{equation} \label{eq:halfspace}
\mathbb{H}_\nu:=\{z\in\mathbb{R}^N:z\cdot\nu<0\},
\end{equation}
where $\nu$ is the exterior unit normal to $\Sigma_\omega$ at $\xi$. Let $u_k(x):=\varepsilon_k^{(2-N)/2}U(\frac{x-\xi}{\varepsilon_k})$, where $U$ is the standard bubble \eqref{eq:bubble}. Then,
\begin{align} 
\lim_{k\to\infty}\int_{\Sigma_\omega}|\nabla u_k|^2=\lim_{k\to\infty}\int_{\Sigma_k}|\nabla U|^2 &=\int_{\mathbb{H}_\nu}|\nabla U|^2=\frac{1}{2N} S^\frac{N}{2},  \label{eq:limit1} \\
\lim_{k\to\infty}\int_{\Sigma_\omega}|u_k|^{2^*}=\lim_{k\to\infty}\int_{\Sigma_k}|U|^{2^*} &=\int_{\mathbb{H}_\nu}|U|^{2^*}=\frac{1}{2N}S^\frac{N}{2}.  \label{eq:limit2}
\end{align}
The function 
$$\widehat{u}_k(x)=u_k(x)-u_k(\varrho x)=\varepsilon_k^{\frac{2-N}{2}}U\left(\frac{x-\xi}{\varepsilon_k}\right)-\varepsilon_k^{\frac{2-N}{2}}U\left(\frac{x-\varrho\xi}{\varepsilon_k}\right)$$
is $\varrho$-equivariant, and from \eqref{eq:inf_rho}, \eqref{eq:limit1} and \eqref{eq:limit2} we obtain 
\begin{equation*}
c^\varrho_{\Sigma_\omega}\leq\lim_{k\to\infty}\frac{1}{N}[Q_{\Sigma_\omega}(\widehat{u}_k)]^\frac{N}{2}=\frac{1}{N}S^\frac{N}{2}=c_\infty.
\end{equation*}
This concludes the proof.
\end{proof}

\begin{lemma} \label{lem:extension}
Given a domain $\Omega$ in $\mathbb{R}^N$ and $\varepsilon>0$, we set $\Omega_{\varepsilon}:=\{\varepsilon^{-1}x:x\in\Omega\}$. If $\partial\Omega$ is Lipschitz continuous, then there exist linear \textbf{extension operators} $P_{\varepsilon}:W^{1,2}(\Omega_{\varepsilon})\to D^{1,2}(\mathbb{R}^N)$ and a positive constant $C$, independent of $\varepsilon$, such that
\begin{enumerate}
\item[$(i)$] $(P_{\varepsilon}\, u)(x)=u(x)$ for every $x\in\Omega_{\varepsilon}$.
\item[$(ii)$] $\int_{\mathbb{R}^N}|\nabla(P_{\varepsilon}\, u)|^2\leq C\int_{\Omega_{\varepsilon}}|\nabla u|^2$.
\item[$(iii)$] $\int_{\mathbb{R}^N}|P_{\varepsilon}\, u|^{2^*}\leq C\int_{\Omega_{\varepsilon}}|u|^{2^*}.$
\item[$(iv)$] If $\Omega$ is $\varrho$-invariant, then $P_{\varepsilon}\, u$ is $\varrho$-equivariant if $u$ is $\varrho$-equivariant.
\end{enumerate}
\end{lemma}

\begin{proof}
The existence of an extension operator $P_{\varepsilon}:W^{1,2}(\Omega_{\varepsilon})\to D^{1,2}(\mathbb{R}^N)$ satisfying $(i)-(iii)$ is well known, and the fact that the constant $C$ does not depend on $\varepsilon$ was proved in \cite[Lemma 2.1]{gp}. To obtain $(iv)$ we replace $P_{\varepsilon}u$ by the function $x\mapsto\frac{1}{2}[(P_{\varepsilon}\,u)(x)-(P_{\varepsilon}\,u)(\varrho x)]$.
\end{proof}

The following proposition describes the behavior of minimizing sequences for $J_{\Lambda_\omega}$ on $\mathcal{N}^\varrho(\Lambda_\omega)$.

\begin{proposition} \label{prop:concentration}
Let $u_k\in\mathcal{N}^\varrho(\Lambda_\omega)$ be such that 
$$J_{\Lambda_\omega}(u_k)\to c^\varrho_{\Lambda_\omega}\qquad\text{and}\qquad J'_{\Lambda_\omega}(u_k)\to 0\text{\; in \;}(V(\Lambda_\omega))'.$$
Then, after passing to a subsequence, one of the following statements holds true:
\begin{itemize}
\item[$(i)$]There exist a sequence of positive numbers $(\varepsilon_k)$, a sequence of points $(\xi_k)$ in $\Gamma_1$ and a function $w\in D^{1,2}(\mathbb{R}^N)$ such that $\varepsilon_k^{-1}\mathrm{dist}(\xi_k,\bar{\omega}\cup\{0\})\to\infty$, $w|_\mathbb{H}$ solves the Neumann problem
\begin{equation} \label{eq:neumann_halfspace}
-\Delta w = |w|^{2^*-2}w,\qquad w\in D^{1,2}(\mathbb{H}),
\end{equation} 
in some half-space $\mathbb{H}$, $J_{\mathbb{H}}(w)=\frac{1}{2}c_\infty$, 
$$\lim_{k\to\infty}\left\|u_k - \varepsilon_k^{\frac{2-N}{2}} w\left(\frac{\,\cdot\,-\xi_k}{\varepsilon_k}\right) + \varepsilon_k^{\frac{2-N}{2}} (w\circ\varrho)\left(\frac{\,\cdot\,-\varrho\xi_k}{\varepsilon_k}\right)\right\|_{\Sigma_\omega}=0,$$
and $c^\varrho_{\Sigma_\omega}=c^\varrho_{\Lambda_\omega}=c_\infty$.
\item[$(ii)$]There exist a sequence of positive numbers $(\varepsilon_k)$ with $\varepsilon_k\to 0$, and a $\varrho$-equivariant solution $w\in D^{1,2}(\Sigma_\omega)$ to the problem \eqref{eq:1} such that
$$\lim_{k\to\infty}\left\|u_k - \varepsilon_k^{\frac{2-N}{2}} w\left(\frac{\,\cdot\,}{\varepsilon_k}\right)\right\|_{\Sigma_\omega}=0,$$
and $J_{\Sigma_\omega}(w)=c^\varrho_{\Sigma_\omega}=c^\varrho_{\Lambda_\omega}\leq c_\infty$.
\end{itemize}
\end{proposition}

\begin{proof}
Since
\begin{equation}
\label{eq:bounded}
\frac{1}{N}\|u_k\|^2_{\Lambda_\omega} = J_{\Lambda_\omega}(u_k) - \frac{1}{2^*}J'_{\Lambda_\omega}(u_k)u_k \leq C+o(1)\|u_k\|_{\Lambda_\omega},
\end{equation}
the sequence $(u_k)$ is bounded and, after passing to a subsequence, $u_k\rightharpoonup u$ weakly in $V(\Lambda_{\omega})$. Then, $J'_{\Lambda_\omega}(u)=0$. Since the problem \eqref{eq:2} does not have a nontrivial solution, we conclude that $u=0$.  

Fix $\delta \in (0,\frac{N}{2}c^\varrho_{\Lambda_\omega})$. As
$$\int_{\Lambda_\omega}|u_k|^{2^*} = N\left(J_{\Lambda_\omega}(u_k) - \frac{1}{2}J'_{\Lambda_\omega}(u_k)u_k\right) \to Nc^{\varrho}_{\Lambda_\omega},$$
there are bounded sequences $(\varepsilon_k)$ in $(0,\infty)$ and $(x_k)$ in $\mathbb{R}^N$ such that, after passing to a subsequence,
\begin{equation*}
  \delta = \sup_{x\in\mathbb{R}^N}\int_{\Lambda_\omega\cap B_{\varepsilon_k}(x)}|u_k|^{2^*}=\int_{\Lambda_\omega\cap B_{\varepsilon_k}(x_k)}|u_k|^{2^*},
\end{equation*}
where $B_r(x):=\{y\in \mathbb{R}^N:|y-x|<r\}$. Note that, as $\delta>0$, we have that $\mathrm{dist}(x_k,\Lambda_\omega)<\varepsilon_k$. We claim that, after passing to a subsequence, there exist $\xi_k\in\bar{\Lambda}_\omega$ and $C_0>0$ such that
\begin{equation} \label{eq:C0}
\varepsilon_k^{-1}|x_k-\xi_k|<C_0\qquad\forall k\in\mathbb{N},
\end{equation}
and one of the following statements holds true:
\begin{itemize}
\item[$(a)$] $\xi_k=0$ for all $k\in\mathbb{N}$.
\item[$(b)$] $\xi_k\in\partial\omega=\overline{\omega}\cap\overline{\Gamma}_1$ for all $k\in\mathbb{N}$.
\item[$(c)$] $\xi_k\in\Gamma_1$ for all $k\in\mathbb{N}$ and $\varepsilon_k^{-1}\mathrm{dist}(\xi_k,\bar{\omega}\cup\{0\})\to \infty$.
\item[$(d)$] $\xi_k\in\omega$ for all $k\in\mathbb{N}$ and $\varepsilon_k^{-1}\mathrm{dist}(\xi_k,\Gamma_1)\to \infty$.
\item[$(e)$] $\xi_k\in\Lambda_\omega$ for all $k\in\mathbb{N}$, $\varepsilon_k^{-1}\mathrm{dist}(\xi_k,\partial\Lambda_\omega)\to\infty$ and, either $\varepsilon_k^{-1}|\xi_k-\varrho\xi_k|\to\infty$, or $\xi_k=\varrho\xi_k$ for all $k\in\mathbb{N}$.
\end{itemize}
This can be seen as follows: If the sequence $(\varepsilon_k^{-1}|x_k|)$ is bounded, we set $\xi_k:=0$. Then, \eqref{eq:C0} and $(a)$ hold true.  If $(\varepsilon_k^{-1}\mathrm{dist}(x_k,\partial\omega))$ is bounded, we take $\xi_k\in\partial\omega$ such that $|x_k-\xi_k|=\mathrm{dist}(x_k,\partial\omega)$. Then, \eqref{eq:C0} and $(b)$ hold true. If both $(\varepsilon_k^{-1}|x_k|)$ and $(\varepsilon_k^{-1}\mathrm{dist}(x_k,\partial\omega))$ are unbounded and $(\varepsilon_k^{-1}\mathrm{dist}(x_k,\Gamma_1))$ is bounded, we take $\xi_k\in\Gamma_1$ with $|x_k-\xi_k|=\mathrm{dist}(x_k,\Gamma_1)$. Then, \eqref{eq:C0} and $(c)$ hold true. If $(\varepsilon_k^{-1}\mathrm{dist}(x_k,\Gamma_1))$ is unbounded and $(\varepsilon_k^{-1}\mathrm{dist}(x_k,\omega))$ is bounded, we take $\xi_k\in\omega$ with $|x_k-\xi_k|=\mathrm{dist}(x_k,\omega)$. Then, \eqref{eq:C0} and $(d)$ hold true. Finally, if $(\varepsilon_k^{-1}\mathrm{dist}(x_k,\partial\Lambda_\omega))$ is unbounded, we set $\xi_k:=\frac{x_k+\varrho x_k}{2}$ if $(\varepsilon_k^{-1}|x_k-\varrho x_k|)$ is bounded and $\xi_k:=x_k$ if $(\varepsilon_k^{-1}|x_k-\varrho x_k|)$ is unbounded. Then, \eqref{eq:C0} and $(e)$ hold true.

Set $C_1:=C_0+1$. Inequality \eqref{eq:C0} yields
\begin{equation} \label{eq:delta}
\delta=\int_{\Lambda_\omega\cap B_{\varepsilon_k}(x_k)}|u_k|^{2^*} \leq \int_{\Lambda_\omega\cap B_{C_1\varepsilon_k}(\xi_k)}|u_k|^{2^*}.
\end{equation}
We consider $u_k$ as a function in $D^{1,2}(\Sigma_\omega)$ via trivial extension, and we define $\widehat{u}_k\in D^{1,2}(\Sigma_\omega)$ as $\widehat{u}_k(z):=\varepsilon_k^{(N-2)/2}u_k(\varepsilon_kz)$. Since $\widehat{u}_k$ is $\varrho$-equivariant, so is its extension $P_{\varepsilon_k}\widehat{u}_k\in D^{1,2}(\mathbb{R}^N)$ given by Lemma \ref{lem:extension}. Let
$$w_k(z):=(P_{\varepsilon_k}\widehat{u}_k)(z+\varepsilon_k^{-1}\xi_k)\in D^{1,2}(\mathbb{R}^N).$$
Then, 
\begin{align} 
w_k(z)=\varepsilon_k^\frac{N-2}{2}u_k(\varepsilon_kz + \xi_k)&\qquad\text{if }\;z\in\Lambda_k:=\varepsilon_k^{-1}(\Lambda_\omega-\xi_k),\label{eq:rescale} \\
w_k\left(z-\varepsilon_k^{-1}\xi_k\right)=-w_k\left(\varrho z-\varepsilon_k^{-1}\xi_k\right)&\qquad\text{for every }z\in \mathbb{R}^N,\label{eq:symmetry}
\end{align}
\begin{equation} \label{eq:nontrivial}
\delta = \sup_{z\in\mathbb{R}^N}\int_{\Lambda_k\cap B_1(z)}|w_k|^{2^*}\leq\int_{\Lambda_k\cap B_{C_1}(0)}|w_k|^{2^*},
\end{equation}
and $(w_k)$ is bounded in $D^{1,2}(\mathbb{R}^N)$. Hence, a subsequence satisfies that $w_k \rightharpoonup w$ weakly in $D^{1,2}(\mathbb{R}^N)$, $w_k \to w$ a.e. in $\mathbb{R}^N$ and $w_k \to w$ strongly in $L^2_{\mathrm{loc}}(\mathbb{R}^N)$. Choosing $\delta$ sufficiently small and using \eqref{eq:nontrivial}, a standard argument shows that $w\neq 0$; see, e.g., \cite[Section 8.3]{w}. Moreover, we have that $\xi_k \to \xi$ and $\varepsilon_k \to 0$, because $u_k \rightharpoonup 0$ weakly in $V(\Lambda_{\omega})$ and $w\neq 0$.

Let $\mathbb{E}$ be the limit of the domains $\Lambda_k$. Since $(w_k)$ is bounded in $D^{1,2}(\mathbb{R}^N)$, using Hölder's inequality we obtain
\begin{align*}
\left|\int_{\mathbb{E}\smallsetminus\Lambda_k}\nabla w_k\cdot\nabla \varphi\right| &\leq C\left(\int_{\mathbb{E}\smallsetminus\Lambda_k}|\nabla \varphi|^2\right)^\frac{1}{2}=o(1),\\
\left|\int_{\mathbb{E}\smallsetminus\Lambda_k}|w_k|^{2^*-2}w_k\varphi\right| &\leq C\left(\int_{\mathbb{E}\smallsetminus\Lambda_k}|\varphi|^{2^*}\right)^{\frac{1}{2^*}}=o(1),
\end{align*}
for every $\varphi\in\mathcal{C}^\infty_c(\mathbb{R}^N)$, and similarly for the integrals over $\Lambda_k\smallsetminus\mathbb{E}$. Therefore, as $w_k \rightharpoonup w$ weakly in $D^{1,2}(\mathbb{E})$, rescaling and using \eqref{eq:rescale} we conclude that
\begin{align}
&\int_\mathbb{E}\nabla w\cdot\nabla \varphi - \int_\mathbb{E}|w|^{2^*-2}w\varphi =\int_\mathbb{E}\nabla w_k\cdot\nabla \varphi - \int_\mathbb{E}|w_k|^{2^*-2}w_k\varphi + o(1)\nonumber \\
&\qquad=\int_{\Lambda_k}\nabla w_k\cdot\nabla\varphi - \int_{\Lambda_k}|w_k|^{2^*-2}w_k\varphi + o(1) \nonumber\\
&\qquad=\int_{\Lambda_\omega}\nabla u_k\cdot\nabla\varphi_k- \int_{\Lambda_\omega}|u_k|^{2^*-2}u_k\varphi_k + o(1),\label{eq:solution}
\end{align}
where $\varphi_k(x):=\varepsilon_k^{(2-N)/2}\varphi(\frac{x-\xi_k}{\varepsilon_k})$. Next, we analyze all possibilities, according to the location of $\xi_k$.
\begin{enumerate}
\item[$(a)$] If $\xi_k=0$ for all $k\in\mathbb{N}$, then $\mathbb{E}=\Sigma_\omega$ and $w_k$ is $\varrho$-equivariant. Hence, $w$ is $\varrho$-equivariant. Let $\varphi\in\mathcal{C}^\infty_c(\mathbb{R}^N)$. Then, $\varphi_k|_{\Lambda_\omega}\in V(\Lambda_\omega)$ for large enough $k$, and from \eqref{eq:solution} we obtain
\begin{align*}
J'_{\Sigma_\omega}(w)[\varphi|_{\Sigma_\omega}]=\int_{\Sigma_\omega}\nabla w\cdot\nabla\varphi - \int_{\Sigma_\omega}|w|^{2^*-2}w\varphi=J'_{\Lambda_\omega}(u_k)[\varphi_k|_{\Lambda_\omega}]=o(1).
\end{align*}
This shows that $w|_{\Sigma_\omega}$ solves \eqref{eq:1}. Therefore,
$$c^\varrho_{\Sigma_\omega}\leq \frac{1}{N}\|w\|^2_{\Sigma_\omega} \leq \liminf_{k\to\infty}\frac{1}{N}\|w_k\|^2_{\Sigma_\omega}=\lim_{k\to\infty}\frac{1}{N}\|u_k\|^2_{\Lambda_\omega}=c^\varrho_{\Lambda_\omega}.$$
Together with Lemma \ref{lem:infimum}, this implies that $J_{\Sigma_\omega}(w)=c^\varrho_{\Sigma_\omega}=c^\varrho_{\Lambda_\omega}\leq c_\infty$ and
$$o(1)=\|w_k-w\|_{\Sigma_\omega}=\left\|u_k-\varepsilon_k^\frac{2-N}{2}w\left(\frac{\cdot}{\varepsilon_k}\right)\right\|_{\Sigma_\omega}.$$
So, in this case, we obtain statement $(ii)$.

\item[$(b)$] If $\xi_k\in\partial\omega$ for all $k\in\mathbb{N}$, then $\mathbb{E}=\mathbb{H}_\xi\cap\mathbb{H}_\nu$, where $\xi=\lim_{k\to\infty}\xi_k$, \;$\nu$ is the exterior unit normal to $\Sigma_\omega$ at $\xi$, and $\mathbb{H}_\xi$ and $\mathbb{H}_\nu$ are half-spaces defined as in \eqref{eq:halfspace}. If $\varphi\in\mathcal{C}^\infty_c(\mathbb{H}_\xi)$, then $\varphi_k|_{\Lambda_\omega}\in V(\Lambda_\omega)$ for large enough $k$, and using \eqref{eq:solution} we conclude that $w|_\mathbb{E}$ solves the mixed boundary value problem
$$-\Delta w = |w|^{2^*-2}w,\qquad w=0\text{ on }\partial\mathbb{E}\cap\partial\mathbb{H}_\xi,\qquad \frac{\partial w}{\partial\nu}=0\text{ on }\partial\mathbb{E}\cap\partial\mathbb{H}_\nu.$$
Since $\xi$ and $\nu$ are orthogonal, extending $w|_\mathbb{E}$ by reflection on $\partial\mathbb{E}\cap\partial\mathbb{H}_\nu$, yields a nontrivial solution to the Dirichlet problem
\begin{equation} \label{eq:dirichlet_halfspace}
-\Delta w = |w|^{2^*-2}w,\qquad w\in D^{1,2}_0(\mathbb{H}_\xi).
\end{equation} 
It is well known that this problem does not have a nontrivial solution, so $(b)$ cannot occur.

\item[$(c)$] If $\xi_k\in\Gamma_1$ for all $k\in\mathbb{N}$ and $\varepsilon_k^{-1}\mathrm{dist}(\xi_k,\bar{\omega}\cup\{0\})\to \infty$, then $\mathbb{E}=\mathbb{H}_\nu$, where $\nu$ is the exterior unit normal to $\Sigma_\omega$ at $\xi=\lim_{k\to\infty}\xi_k$. Using \eqref{eq:solution} we conclude that $w|_{\mathbb{H}_\nu}$ solves the Neumann problem \eqref{eq:neumann_halfspace} in $\mathbb{H}_\nu$. Since  $\varepsilon_k^{-1}|\xi_k|\to\infty$, we have that $\varepsilon_k^{-1}|\xi_k-\varrho\xi_k|\to\infty$. Therefore, 
$$w_k-(w\circ\varrho)(\,\cdot\,+\varepsilon_k^{-1}(\xi_k-\varrho\xi_k))\rightharpoonup w\qquad\text{weakly in }D^{1,2}(\mathbb{R}^{N}).$$ 
Note also that $w_k\circ\varrho\rightharpoonup w\circ\varrho$ weakly in $D^{1,2}(\mathbb{R}^{N})$. Using these facts and performing suitable rescalings and translations we obtain
\begin{align*}
&\left\|u_k - \varepsilon_k^{\frac{2-N}{2}} w\left(\frac{\,\cdot\,-\xi_k}{\varepsilon_k}\right) + \varepsilon_k^{\frac{2-N}{2}} (w\circ\varrho)\left(\frac{\,\cdot\,-\varrho\xi_k}{\varepsilon_k}\right)\right\|_{\Sigma_\omega}^2\\
&=\left\|\widehat{u}_k - w(\,\cdot\,-\varepsilon_k^{-1}\xi_k) + (w\circ\varrho)(\,\cdot\,-\varepsilon_k^{-1}\varrho\xi_k)\right\|_{\Sigma_\omega}^2\\
&=\left\|w_k - w + (w\circ\varrho)\left(\,\cdot\,+\varepsilon_k^{-1}(\xi_k-\varrho\xi_k)\right)\right\|_{\Sigma_\omega-\varepsilon_k^{-1}\xi_k}^2\\
&=\left\|w_k + (w\circ\varrho)\left(\,\cdot\,+\varepsilon_k^{-1}(\xi_k-\varrho\xi_k)\right)\right\|_{\Sigma_\omega-\varepsilon_k^{-1}\xi_k}^2 - \|w\|^2_{\mathbb{H}_\nu} + o(1)\\
&=\left\|-w_k\circ\varrho + w\circ\varrho\right\|_{\Sigma_\omega-\varepsilon_k^{-1}\varrho\xi_k}^2 - \|w\|^2_{\mathbb{H}_\nu} + o(1)\\
&=\left\|\widehat{u}_k\right\|_{\Sigma_\omega}^2 - 2\|w\|^2_{\mathbb{H}_\nu} + o(1)\\
&=\left\|u_k\right\|_{\Lambda_\omega}^2 - 2\|w\|^2_{\mathbb{H}_\nu} + o(1)= Nc_{\Lambda_\omega}^\varrho - 2\|w\|^2_{\mathbb{H}_\nu}+o(1).
\end{align*}
Since $J_{\mathbb{H}_\nu}(w)=\frac{1}{N}\|w\|^2_{\mathbb{H}_\nu}\geq \frac{1}{2}c_\infty$, applying Lemma \ref{lem:infimum} we conclude that $J_{\mathbb{H}_\nu}(w)=\frac{1}{2}c_\infty$, \,$c_{\Sigma_\omega}^\varrho=c_{\Lambda_\omega}^\varrho=c_\infty$, and 
$$\lim_{k\to\infty}\left\|u_k - \varepsilon_k^{\frac{2-N}{2}} w\left(\frac{\,\cdot\,-\xi_k}{\varepsilon_k}\right) + \varepsilon_k^{\frac{2-N}{2}} (w\circ\varrho)\left(\frac{\,\cdot\,-\varrho\xi_k}{\varepsilon_k}\right)\right\|_{\Sigma_\omega}^2=0.$$
So, in this case we obtain statement $(i)$.

\item[$(d)$] If $\xi_k\in\omega$ for all $k\in\mathbb{N}$ and $\varepsilon_k^{-1}\mathrm{dist}(\xi_k,\Gamma_1)\to \infty$, then $\mathbb{E}=\mathbb{H}_\xi$ and using \eqref{eq:solution} we conclude that $w|_{\mathbb{H}_\xi}$ solves the Dirichlet problem \eqref{eq:dirichlet_halfspace}. So this case does not occur.

\item[$(e)$] If $\xi_k\in\Lambda_\omega$ for all $k\in\mathbb{N}$ and $\varepsilon_k^{-1}\mathrm{dist}(\xi_k,\partial\Lambda_\omega)\to\infty$, then $\mathbb{E}=\mathbb{R}^N$ and $w$ solves the problem \eqref{eq:entire}. 
If $\rho\xi_k=\xi_k$ for every $k$, then $w_k$ is $\varrho$-equivariant, and so is $w$. Since $w$ is a sign-changing solution to \eqref{eq:entire} we have that
$$2c_\infty < \frac{1}{N}\|w\|^2_{\mathbb{R}^N} \leq \lim_{k\to\infty}\frac{1}{N}\|w_k\|^2_{\mathbb{R}^N} = \lim_{k\to\infty}\frac{1}{N}\|u_k\|^2_{\Lambda_\omega} = c^{\varrho}_{\Lambda_\omega},$$
contradicting Lemma \ref{lem:infimum}. On the other hand, if $\varepsilon_k^{-1}|\varrho\xi_k-\xi_k|\to\infty$, then, arguing as in case $(c)$, we conclude that
$$2c_\infty \leq \frac{2}{N}\|w\|^2_{\mathbb{R}^N}\leq \lim_{k\to\infty}\frac{1}{N}\|w_k\|^2_{\mathbb{R}^N} = \lim_{k\to\infty}\frac{1}{N}\|u_k\|^2_{\Lambda_\omega} = c^{\varrho}_{\Lambda_\omega},$$
contradicting Lemma \ref{lem:infimum} again. So $(e)$ cannot occur.
\end{enumerate}
We are left with $(a)$ and $(c)$. This concludes the proof.
\end{proof}

Proposition \ref{prop:concentration} immediately yields the following result.

\begin{corollary} \label{cor:existence_nodal}
If $c^{\varrho}_{\Sigma_\omega}<c_\infty$, then the problem \eqref{eq:1} has a $\varrho$-equivariant least energy solution in $D^{1,2}(\Sigma_\omega)$.
\end{corollary}

Equality is not enough, as the following example shows. Set
$$\mathbb{S}^{N-1}_+:=\{(x_1,\ldots,x_N)\in\mathbb{S}^{N-1}:x_N>0\}.$$

\begin{example}
If $\omega=\mathbb{S}^{N-1}_+$, then problem \eqref{eq:1} does not have a $\varrho$-equivariant least energy solution in $D^{1,2}(\Sigma_\omega)$.
\end{example}

\begin{proof}
$\Sigma_\omega$ is the upper half-space $\mathbb{R}^N_+:=\{(x_1,\ldots,x_N)\in\mathbb{R}^N:x_N>0\}$. If $u$ were a $\varrho$-equivariant least energy solution to \eqref{eq:1} in $\mathbb{R}^N_+$ then, extending $u$ by reflection on $\partial(\mathbb{R}^N_+)$, would yield a sign-changing solution $\widetilde{u}$ to the problem \eqref{eq:entire} in $\mathbb{R}^N$ with $J_{\mathbb{R}^N}(\widetilde{u})\leq 2c_\infty$. But the energy of any sign-changing solution to \eqref{eq:entire} is $>2c_\infty$; see \cite{we}.
\end{proof}

The following local geometric condition guarantees the existence of a minimizer. It was introduced by Adimurthi and Mancini in \cite{am}.

\begin{definition} \label{def:pc}
A point $\xi\in\partial\omega$ is a \textbf{point of convexity of $\Sigma_\omega$ of radius} $r>0$ if \,$B_r(\xi)\cap\Sigma_\omega\subset\mathbb{H}_\nu$\, and the mean curvature of $\partial\Sigma_\omega$ at $\xi$ with respect to the exterior unit normal $\nu$ at $\xi$ is positive.
\end{definition}

As in \cite{am} we make the convention that the curvature of a geodesic in $\partial\Sigma_\omega$ is positive at $\xi$ if it curves away from the exterior unit normal $\nu$. The half-space $\mathbb{H}_\nu$ is defined as in \eqref{eq:halfspace}. Examples of cones having a point of convexity are given as follows.

\begin{proposition} 
If $\overline{\omega}\subset\mathbb{S}^{N-1}_+$, then $\Sigma_\omega$ has a point of convexity.
\end{proposition}

\begin{proof}
Let $\beta$ be the smallest geodesic ball in $\mathbb{S}^{N-1}$, centered at the north pole $(0,\ldots,0,1)$, which contains $\omega$. Then, $\partial\omega\cap\partial\beta\neq\emptyset$ and $\overline{\beta}\subset\mathbb{S}^{N-1}_+$. Hence, every point on $\partial\beta$ is a point of convexity of $\Sigma_\beta$. As $\omega\subset\beta$, we have that any point $\xi\in\partial\omega\cap\partial\beta$ is a point of convexity of $\Sigma_\omega$.
\end{proof}

\begin{theorem} \label{thm:existence_nodal}
If $\Sigma_\omega$ has a point of convexity, then $c^{\varrho}_{\Sigma_\omega}<c_\infty$. Consequently, the problem \eqref{eq:1} has a $\varrho$-equivariant least energy solution in $D^{1,2}(\Sigma_\omega)$. This solution is nonradial and changes sign.
\end{theorem}

\begin{proof}
Let $\xi\in\partial\omega$ be a point of convexity of $\Sigma_\omega$ of radius $r>0$. It is shown in \cite[Lemma 2.2]{am} that, after fixing $r$ small enough and a radial cut-off function $\psi\in\mathcal{C}^\infty_c(\mathbb{R}^N)$ with $\psi(x)=1$ if $|x|\leq \frac{r}{4}$ and $\psi(x)=0$ if $|x|\geq \frac{r}{2}$, the function $u_{\varepsilon,\xi}(x):=\psi(x-\xi)\varepsilon^{(2-N)/2}U(\varepsilon^{-1}(x-\xi))$, with $U$ as in \eqref{eq:bubble}, satisfies
\begin{equation} \label{eq:am}
Q_{\Sigma_\omega}(u_{\varepsilon,\xi})=
\begin{cases}
\frac{S}{2^{2/N}}-d_NH_\omega(\xi)S\,\varepsilon\ln(\varepsilon^{-2})+O(\varepsilon) &\text{if }N=3,\\
\frac{S}{2^{2/N}}-d_NH_\omega(\xi)S\,\varepsilon+O(\varepsilon^2\ln(\varepsilon^{-2})) &\text{if }N\geq 4,
\end{cases}
\end{equation}
where $d_N$ is a positive constant depending only on $N$ and $H_\omega(\xi)$ is the mean curvature of $\partial\Sigma_\omega$ at $\xi$. Hence, for $\varepsilon$ small enough,
$$J_{\Sigma_\omega}(t_{\varepsilon,\xi}u_{\varepsilon,\xi})=\frac{1}{N}[Q_{\Sigma_\omega}(u_{\varepsilon,\xi})]^\frac{N}{2}<\frac{1}{2N}S^\frac{N}{2}=\frac{1}{2} c_\infty,$$
where $t_{\varepsilon,\xi}>0$ is such that $t_{\varepsilon,\xi}u_{\varepsilon,\xi}\in\mathcal{N}(\Sigma_\omega)$; see \eqref{eq:JQ}. Choosing $r$ so that $B_r(\xi)\cap B_r(\varrho\xi)=\emptyset$ we conclude that \,$t_{\varepsilon,\xi}(u_{\varepsilon,\xi}-u_{\varepsilon,\xi}\circ\varrho)\in\mathcal{N}^\varrho(\Sigma_\omega)$\, and
$$c^{\varrho}_{\Sigma_\omega}\leq J_{\Sigma_\omega}(t_{\varepsilon,\xi}(u_{\varepsilon,\xi}-u_{\varepsilon,\xi}\circ\varrho))<c_\infty.$$
The existence of a $\varrho$-equivariant least energy solution to \eqref{eq:1} follows from Corollary \ref{cor:existence_nodal}.
\end{proof}

\section{A positive nonradial solution}
\label{sec:positive}

In this section $\omega$ is not assumed to have any symmetries. 

We are interested in positive solutions to the problem \eqref{eq:1}. Note that this problem has always a positive \emph{radial} solution given by the restriction to $\Sigma_\omega$ of the standard bubble $U$ defined in \eqref{eq:bubble}. The question we wish to address in this section is whether problem \eqref{eq:1} has a positive \emph{nonradial} solution. 

Recall the notation introduced in Section \ref{sec:nodal} and set
\begin{equation*}
c_{\Sigma_\omega}:=\inf_{u\in\mathcal{N}(\Sigma_\omega)}J_{\Sigma_\omega}(u)=\inf_{u\in D^{1,2}(\Sigma_\omega)\smallsetminus\{0\}}\frac{1}{N}[Q_{\Sigma_\omega}(u)]^\frac{N}{2},
\end{equation*} 
$$\mathcal{N}(\Lambda_\omega):=\mathcal{N}(\Sigma_\omega)\cap V(\Lambda_{\omega})\qquad\text{and}\qquad c_{\Lambda_\omega}:=\inf_{u\in\mathcal{N}(\Lambda_\omega)}J_{\Lambda_\omega}(u).$$ 
It is shown in \cite[Theorem 2.1]{lpt} that $c_{\Lambda_\omega}>0$. As in Lemma \ref{lem:infimum} one shows that $c_{\Sigma_\omega}=c_{\Lambda_\omega}\leq\frac{1}{2} c_\infty$. We start by describing the behavior of minimizing sequences for $J_{\Lambda_\omega}$ on $\mathcal{N}(\Lambda_\omega)$.

\begin{proposition} \label{prop:concentration_positive}
Let $u_k\in\mathcal{N}(\Lambda_\omega)$ be such that 
$$J_{\Lambda_\omega}(u_k)\to c_{\Lambda_\omega}\qquad\text{and}\qquad J'_{\Lambda_\omega}(u_k)\to 0\text{\; in \;}(V(\Lambda_\omega))'.$$
Then, after passing to a subsequence, one of the following statements holds true:
\begin{itemize}
\item[$(i)$]There exist a sequence of positive numbers $(\varepsilon_k)$, a sequence of points $(\xi_k)$ in $\Gamma_1$ and a function $w\in D^{1,2}(\mathbb{R}^N)$ such that $\varepsilon_k^{-1}\mathrm{dist}(\xi_k,\bar{\omega}\cup\{0\})\to\infty$, $w|_\mathbb{H}$ solves the Neumann problem
\begin{equation*}
-\Delta w = |w|^{2^*-2}w,\qquad w\in D^{1,2}(\mathbb{H}),
\end{equation*} 
in some half-space $\mathbb{H}$, $J_{\mathbb{H}}(w)=\frac{1}{2}c_\infty$, 
$$\lim_{k\to\infty}\left\|u_k - \varepsilon_k^{\frac{2-N}{2}} w\left(\frac{\,\cdot\,-\xi_k}{\varepsilon_k}\right)\right\|_{\Sigma_\omega}=0,$$
and $c_{\Sigma_\omega}=c_{\Lambda_\omega}=\frac{1}{2} c_\infty$.
\item[$(ii)$]There exist a sequence of positive numbers $(\varepsilon_k)$ with $\varepsilon_k\to 0$ and a solution $w\in D^{1,2}(\Sigma_\omega)$ to the problem \eqref{eq:1} such that
$$\lim_{k\to\infty}\left\|u_k - \varepsilon_k^{\frac{2-N}{2}} w\left(\frac{\,\cdot\,}{\varepsilon_k}\right)\right\|_{\Sigma_\omega}=0,$$
and $J_{\Sigma_\omega}(w)=c_{\Sigma_\omega}=c_{\Lambda_\omega}\leq\frac{1}{2} c_\infty$.
\end{itemize}
\end{proposition}

\begin{proof}
The proof is similar, but simpler than that of Proposition \ref{prop:concentration}.
\end{proof}

The following statement is an immediate consequence of this proposition.

\begin{corollary} \label{cor:existence_positive}
If $c_{\Sigma_\omega}<\frac{1}{2}c_\infty$, then the problem \eqref{eq:1} has a positive least energy solution in $D^{1,2}(\Sigma_\omega)$.
\end{corollary}

\begin{theorem} \label{thm:existence_positive}
If $\Sigma_\omega$ has a point of convexity, then $c_{\Sigma_\omega}<\frac{1}{2} c_\infty$. Consequently, the problem \eqref{eq:1} has a positive least energy solution in $D^{1,2}(\Sigma_\omega)$. 
\end{theorem}

\begin{proof}
The proof is similar to that of Theorem \ref{thm:existence_nodal}.
\end{proof}

Let $D^{1,2}_\mathrm{rad}(\Sigma_\omega)$ be the subspace of radial functions in $D^{1,2}(\Sigma_\omega)$, and define $\mathcal{N}^\mathrm{rad}(\Sigma_\omega):=\mathcal{N}(\Sigma_\omega)\cap D^{1,2}_\mathrm{rad}(\Sigma_\omega)$ and
$$c_{\Sigma_\omega}^\mathrm{rad}:=\inf_{u\in\mathcal{N}^\mathrm{rad}(\Sigma_\omega)}J_{\Sigma_\omega}(u)=\inf_{u\in D^{1,2}_\mathrm{rad}(\Sigma_\omega)\smallsetminus\{0\}}\frac{1}{N}[Q_{\Sigma_\omega}(u)]^\frac{N}{2}.$$ 
It was shown in \cite[Theorem 2.4]{lpt} that, if $\Sigma_\omega$ is convex, then $c_{\Sigma_\omega}^\mathrm{rad}=c_{\Sigma_\omega}$ and the only positive minimizers are the restrictions of the rescalings
\begin{equation} \label{eq:bubble_rescaled}
U_\varepsilon(x)=a_N\left(\frac{\varepsilon}{\varepsilon^2+|x|^2}\right)^\frac{N-2}{2},\qquad\varepsilon>0,
\end{equation}
of the standard bubble to $\Sigma_\omega$. In fact, the proof of \cite[Theorem 2.4]{lpt} shows that these are the only positive solutions of \eqref{eq:1} in a convex cone. Moreover, the following statement holds true.

\begin{proposition} \label{prop:radial_energy}
For any cone $\Sigma_\omega$, the restrictions to $\Sigma_\omega$ of the functions $U_\varepsilon$ defined in \eqref{eq:bubble_rescaled} are minimizers of $J_{\Sigma_\omega}$ on $\mathcal{N}^\mathrm{rad}(\Sigma_\omega)$. These are the only nontrivial radial solutions to \eqref{eq:1}, up to sign. Moreover,
$$c_{\Sigma_\omega}^\mathrm{rad} = b_N|\Lambda_\omega|,\qquad\text{where }b_N=\frac{c_\infty}{|B_1(0)|}$$
and $|X|$ is the Lebesgue measure of $X$. In particular, $c_{\Sigma_\omega}^\mathrm{rad}$ increases with $|\Lambda_\omega|$.
\end{proposition}

\begin{proof}
A radial function $u$ solves \eqref{eq:1} in $\Sigma_\omega$ if and only if the function $\bar{u}$ given by $\bar{u}(r):=u(x)$ with $r=\|x\|$ solves
$$\frac{\mathrm{d}}{\mathrm{d}r}(r^{N-1}\bar{u}'(r))=r^{N-1}|\bar{u}(r)|^{N-2}\bar{u}(r)\text{ in }(0,\infty),\quad\bar{u}(0)=u(0),\quad\bar{u}'(0)=0.$$
This last problem does not depend on $\omega$. It is well known that, up to sign, the functions $U_\varepsilon$ are the only nontrivial radial solutions to the problem \eqref{eq:entire} in $\mathbb{R}^N=\Sigma_{\mathbb{S}^{N-1}}$. Hence, their restrictions to $\Sigma_\omega$ are the only nontrivial radial solutions to \eqref{eq:1}.

As in Lemma \ref{lem:infimum} one shows that $c_{\Sigma_\omega}^\mathrm{rad}=c_{\Lambda_\omega}^\mathrm{rad}:=\inf_{u\in\mathcal{N}^\mathrm{rad}(\Lambda_\omega)}J_{\Lambda_\omega}(u)$. For $u\in V_\mathrm{rad}(\Lambda_\omega):=D^{1,2}_\mathrm{rad}(\Lambda_\omega)\cap V(\Lambda_\omega), u\neq 0$, we have that
$$Q_{\Lambda_\omega}(u)=\frac{\int_{\Lambda_\omega}|\nabla u|^2}{\left(\int_{\Lambda_\omega}|u|^{2^*}\right)^{2/2^*}}=\frac{N|\Lambda_\omega|\int_0^1|\bar{u}'(r)|^2r^{N-1}\mathrm{d}r}{\left(N|\Lambda_\omega|\int_0^1|\bar{u}(r)|^{2^*}r^{N-1}\mathrm{d}r\right)^{2/2^*}}.$$
Therefore,
\begin{align*}
c_{\Lambda_\omega}^\mathrm{rad} &=\inf_{u\in V_\mathrm{rad}(\Lambda_\omega)\smallsetminus\{0\}}\frac{1}{N}[Q_{\Lambda_\omega}(u)]^\frac{N}{2}\\
&=\inf_{u\in V_\mathrm{rad}(\Lambda_\omega)\smallsetminus\{0\}}\frac{\int_0^1|\bar{u}'(r)|^2r^{N-1}\mathrm{d}r}{\left(\int_0^1|\bar{u}(r)|^{2^*}r^{N-1}\mathrm{d}r\right)^{2/2^*}}|\Lambda_\omega|=:b_N|\Lambda_\omega|.
\end{align*}
The same formula holds true when we replace $\omega$ by $\mathbb{S}^{N-1}$. In this case, the left-hand side is $c_\infty$. Hence, $b_N=\frac{c_\infty}{|B_1(0)|}$, as claimed.
\end{proof}

\begin{corollary} \label{cor:nonradial}
If $\Sigma_\omega$ has a point of convexity and $|\Lambda_\omega|\geq\frac{1}{2}|B_1(0)|$, then
\begin{itemize}
\item[$(i)$]the problem \eqref{eq:1} has a positive least energy solution in $D^{1,2}(\Sigma_\omega)$,
\item[$(ii)$]every least energy solution of \eqref{eq:1} is nonradial.
\end{itemize}
\end{corollary}

\begin{proof}
From Theorem \ref{thm:existence_positive} and Proposition \ref{prop:radial_energy} we get that $c_{\Sigma_\omega}$ is attained and
$$c_{\Sigma_\omega}<\frac{1}{2} c_\infty=c_{\mathbb{R}^N_+}^\mathrm{rad}=\frac{b_N}{2}|B_1(0)|\leq b_N|\Lambda_\omega|=c_{\Sigma_\omega}^\mathrm{rad},$$
where $\mathbb{R}^N_+:=\{(x_1,\ldots,x_N)\in\mathbb{R}^N:x_N>0\}$. So every least energy solution is nonradial.
\end{proof}

Note that the hypothesis that $|\Lambda_\omega|\geq\frac{1}{2}|B_1(0)|$ implies that $\Sigma_\omega$ is not convex.

A closer look at the estimate \eqref{eq:am} allows to refine Corollary \ref{cor:nonradial} and to produce examples of cones $\Sigma_\omega$ with $|\Lambda_\omega|<\frac{1}{2}|B_1(0)|$ for which the problem \eqref{eq:1} has a positive nonradial solution.

To this end, we fix a smooth domain $\omega_0$ in $\mathbb{S}^{N-1}$ for which $\Sigma_{\omega_0}$ has a point of convexity $\xi\in\partial\omega_0$ of radius $r>0$, and we define
\begin{align*}
\ell(\omega_0,\xi,r):=\{\omega:\;&\omega \text{ is a smooth domain in }\mathbb{S}^{N-1},\;B_r(\xi)\cap\Sigma_{\omega_0}\subset B_r(\xi)\cap\Sigma_{\omega}\\
&\text{and }\;\mathrm{dist}(B_r(\xi)\cap\Sigma_{\omega_0},\; B_r(\xi)\cap(\Sigma_{\omega}\smallsetminus\Sigma_{\omega_0}))>0\}.
\end{align*}
Then, we have the following result.

\begin{theorem} \label{thm:nonradial}
There exists $\alpha_\xi\in (0,\frac{1}{2}|B_1(0)|)$, depending only on $B_r(\xi)\cap\Sigma_{\omega_0}$, such that, for every $\omega\in \ell(\omega_0,\xi,r)$ with $|\Lambda_\omega|>\alpha_\xi$, the following statements hold true:
\begin{itemize}
\item[$(i)$]the problem \eqref{eq:1} has a positive least energy solution in $D^{1,2}(\Sigma_\omega)$,
\item[$(ii)$]every least energy solution of \eqref{eq:1} is nonradial,
\item[$(iii)$]$\Sigma_\omega$ is not convex.
\end{itemize}
\end{theorem}

\begin{proof}
Recall that the functions $u_{\varepsilon,\xi}$, introduced in the proof of Theorem \ref{thm:existence_nodal}, vanish outside the ball $B_{r/2}(0)$. Moreover, the value $Q_{\Sigma_{\omega_0}}(u_{\varepsilon,\xi})$ and the estimate \eqref{eq:am} depend only on the value of $u_{\varepsilon,\xi}$ in $B_r(\xi)\cap\Sigma_{\omega_0 }$. We fix $\varepsilon_0>0$ small enough so that
$$Q_\xi:=Q_{\Sigma_{\omega_0}}(u_{\varepsilon_0,\xi})<\frac{S}{2^{2/N}},$$
and we set $\alpha_\xi:=\frac{1}{Nb_N}Q_\xi^{N/2}$ with $b_N$ as in Proposition \ref{prop:radial_energy}. Then,
$$\alpha_\xi<\frac{1}{2Nb_N}S^\frac{N}{2}=\frac{1}{2}|B_1(0)|.$$
Given $\omega\in \ell(\omega_0,\xi,r)$, we fix a function $\widehat{u}_{\varepsilon_0,\xi}\in\mathcal{C}_c^\infty(B_r(0))$ such that $\widehat{u}_{\varepsilon_0,\xi}(x)=u_{\varepsilon_0,\xi}(x)$ if $x\in B_r(\xi)\cap\Sigma_{\omega_0}$ and $\widehat{u}_{\varepsilon_0,\xi}(x)=0$ if $x\in B_r(\xi)\cap(\Sigma_{\omega}\smallsetminus\Sigma_{\omega_0})$. So, if $|\Lambda_\omega|>\alpha_\xi$, we have that
$$c_{\Sigma_\omega}\leq\frac{1}{N}[Q_{\Sigma_\omega}(\widehat{u}_{\varepsilon_0,\xi})]^\frac{N}{2}=\frac{1}{N}Q_\xi^\frac{N}{2}=b_N\alpha_\xi<b_N|\Lambda_\omega|=c_{\Sigma_\omega}^\mathrm{rad}.$$
Note that $\xi$ is a point of convexity of $\omega$. Hence, by Theorem  \ref{thm:existence_positive} and the previous inequality, $c_{\Sigma_\omega}$ is attained at a nonradial solution of \eqref{eq:1}. 
Finally, recall that, if $\Sigma_\omega$ were convex, then $c_{\Sigma_\omega}=c_{\Sigma_\omega}^\mathrm{rad}$; see \cite[Theorem 2.4]{lpt}. This completes the proof.
\end{proof}

\begin{corollary}
There exists a smooth domain $\omega\subset\mathbb{S}^{N-1}_+$ such that the problem \eqref{eq:1} has a positive nonradial solution in $\Sigma_\omega$.
\end{corollary}

\begin{proof}
Let $\omega_0$ be the geodesic ball in $\mathbb{S}^{N-1}$ of radius $\pi/4$ centered at the north pole and let $\xi$ be any point on $\partial\omega_0$. Fix $r>0$ such that $B_r(\xi)\cap\mathbb{S}^{N-1}\subset\mathbb{S}^{N-1}_+$. Clearly, $\xi$ is a point of convexity of $\Sigma_{\omega_0}$ of radius $r$, so we may fix $\alpha_\xi>0$ as in Theorem \ref{thm:nonradial}. As $\alpha_\xi < \frac{1}{2}|B_1(0)|$, there exists $\omega\in \ell(\omega_0,\xi,r)$ with $\omega\subset\mathbb{S}^{N-1}_+$ and $|\Lambda_\omega|>\alpha_\xi$. Now, Theorem \ref{thm:nonradial} yields a positive nonradial solution to problem \eqref{eq:1} in $\Sigma_\omega$.
\end{proof}

\begin{remark}
Let $\omega$ be such that $\Sigma_\omega$ is convex. Then, every point $\xi\in\partial\omega$ is a point of convexity of radius $r$ for any $r>0$. Fix $r=1$, and fix $\varepsilon>0$ such that 
$$Q_\xi:=Q_{\Sigma_{\omega}}(u_{\varepsilon,\xi})<\frac{S}{2^{2/N}}\qquad\forall\xi\in\partial\omega.$$
Now, define $\alpha_\xi:=\frac{1}{Nb_N}Q_\xi^{N/2}$, as in \emph{Theorem} \ref{thm:nonradial}. Since $\Sigma_\omega$ is convex, we must have that
$$|\Lambda_\omega|\leq\alpha_\xi=\frac{|B_1(0)|}{S^{N/2}}Q_\xi^{N/2},\qquad\forall\xi\in\partial\omega,$$
where the equality follows from the definition of $b_N$; see \emph{Proposition} \ref{prop:radial_energy}. Hence, for any convex cone $\Sigma_\omega$, we obtain the upper bound
$$|\Lambda_\omega|\leq\frac{|B_1(0)|}{S^{N/2}}\min_{\xi\in\partial\omega}Q_\xi$$
for the measure of $\Lambda_\omega$, which is given in terms of the Sobolev constant and the local energy of the standard bubbles.
\end{remark}

 \vspace{15pt}

\begin{flushleft}
\textbf{Mónica Clapp}\\
Instituto de Matemáticas\\
Universidad Nacional Autónoma de México\\
Circuito Exterior, Ciudad Universitaria\\
04510 Coyoacán, CDMX\\
Mexico\\
\texttt{monica.clapp@im.unam.mx} \vspace{10pt}

\textbf{Filomena Pacella}\\
Dipartimento di Matematica\\
Sapienza Università di Roma\\
P.le. Aldo Moro 2\\
00185 Roma\\
Italy\\
\texttt{pacella@mat.uniroma1.it}

\end{flushleft}

\end{document}